\documentclass[12pt]{amsart}

\usepackage{amssymb,amsbsy,amsthm,amsmath,graphicx,mathrsfs,epsfig}
\usepackage{mathabx}

\usepackage{times}

\newtheorem{thm}{Theorem}[section]
\newtheorem*{thm*}{Theorem}
\newtheorem{lem}[thm]{Lemma}
\newtheorem{prop}[thm]{Proposition}
\newtheorem{cor}[thm]{Corollary}

\newtheorem{dfn}[thm]{Definition}

\newtheorem{ques}[thm]{Question}
\newtheorem{ex}[thm]{Example}

\theoremstyle{remark}
\newtheorem*{rmk}{Remark}

\newcommand{\bs}[1]{\boldsymbol{#1}}
\renewcommand{\bf}[1]{\mathbf{#1}}
\renewcommand{\rm}[1]{\mathrm{#1}}
\newcommand{\cal}[1]{\mathcal{#1}}
\renewcommand{\frak}[1]{\mathfrak{#1}}

\newcommand{\bbN}{\mathbb{N}}

\newcommand{\bbZ}{\mathbb{Z}}

\newcommand{\rmH}{\mathrm{H}}

\newcommand{\T}{\mathrm{T}}

\newcommand{\rmh}{\mathrm{h}}

\newcommand{\m}{\mathrm{m}}

\renewcommand{\P}{\mathcal{P}}
\newcommand{\Q}{\mathcal{Q}}

\newcommand{\X}{\mathcal{X}}

\renewcommand{\O}{\Omega}
\renewcommand{\S}{\Sigma}

\newcommand{\eps}{\varepsilon}

\newcommand{\s}{\sigma}
\renewcommand{\phi}{\varphi}
\renewcommand{\k}{\kappa}

\newcommand{\ul}[1]{\underline{#1}}

\newcommand{\fin}{\nolinebreak\hspace{\stretch{1}}$\lhd$}

\renewcommand{\to}{\longrightarrow}

\renewcommand{\t}{\widetilde}
\newcommand{\lws}{\stackrel{\rm{lw}\ast}{\to}}

\newcommand{\lw}{\stackrel{\rm{lw}\ast}{\to}}
\newcommand{\lde}{\stackrel{\rm{lde}}{\to}}

\renewcommand{\Pr}{\mathrm{Prob}}
\newcommand{\cov}{\mathrm{cov}}

\begin{document}


\title[An AEP for measures on model spaces]{An asymptotic equipartition property for measures on model spaces}
\author{Tim Austin}
\address{Courant Institute of Mathematical Sciences, New York University, 251 Mercer St, New York, NY 10012}
\email{tim@cims.nyu.edu}
\urladdr{www.cims.nyu.edu/$\sim$tim}



\begin{abstract}
Let $G$ be a sofic group, and let $\S = (\s_n)_{n\geq 1}$ be a sofic approximation to it.  For a probability-preserving $G$-system, a variant of the sofic entropy relative to $\S$ has recently been defined in terms of sequences of measures on its model spaces that `converge' to the system in a certain sense.  Here we prove that, in order to study this notion, one may restrict attention to those sequences that have the asymptotic equipartition property.  This may be seen as a relative of the Shannon--McMillan theorem in the sofic setting.

We also give some first applications of this result, including a new formula for the sofic entropy of a $(G\times H)$-system obtained by co-induction from a $G$-system, where $H$ is any other infinite sofic group.
\end{abstract}

\maketitle

\section{Introduction}

Let $G$ be a countable sofic group, and let
\[\S = \big(\s_n:G\to \rm{Sym}(V_n)\big)_{n\geq 1}\]
be a sofic approximation to it.

A \textbf{$G$-system} is a triple $(X,\mu,T)$ in which $(X,\mu)$ is a standard probability space and $T = (T^g)_{g\in G}$ is a $\mu$-preserving measurable action of $G$ on $X$.  Using the sofic approximation $\S$, one can construct an isomorphism-invariant of $(X,\mu,T)$ called its sofic entropy.  It is denoted by $\rmh_\S(\mu,T)$.  It was introduced by Lewis Bowen in~\cite{Bowen10} under some extra conditions and then made fully general in~\cite{KerLi11b,KerLi13}.  It solves various classic problems in ergodic theory, such as distinguishing Bernoulli shifts of unequal base entropy over many natural discrete groups~\cite{Bowen10,KerLi11a}.  Bowen's recent survey~\cite{Bowen--survey} gives a thorough introduction to this invariant and explains many directions of current research.

A \textbf{$G$-process} is a $G$-system of the form $(\X^G,\mu,S)$, where $\X$ is another standard measurable space, $S$ is the left coordinate-shift action of $G$ on $\X^G$, and $\mu \in \Pr(\X^G)$ is shift-invariant.  A \textbf{metric $G$-process} is a quadruple $(\X^G,\mu,S,d)$ consisting of a $G$-process and a compact metric $d$ on $\X$ which generates the $\s$-algebra of $\X$ as its Borel sets.

Following~\cite{Aus--soficentadd}, we use a definition of sofic entropy which initially applies only to metric $G$-processes, and is then extended to other systems by isomorphism-invariance.  For a metric $G$-process $(\X^G,\mu,S,d)$, sofic entropy is defined in terms the spaces $\X^{V_n}$, and is abbreviated to $\rmh_\S(\mu)$.  For each element of $\X^{V_n}$, one defines an empirical distribution on $\X^G$ by a kind of `pulling back' under $\s_n$.  An element of $\X^{V_n}$ is a `good model' for $\mu$ if its empirical distribution is close to $\mu$ in the weak$^\ast$ topology of $\Pr(\X^G)$, where $\X^G$ has the product topology resulting from $d$. Loosely, the sofic entropy $\rmh_\S(\mu)$ is defined as the exponential growth rate of the sizes of the sets of these good models.  Here `size' refers to cardinality if $\X$ is finite and to small-radius covering numbers if $(\X,d)$ is compact but not finite.

Starting with~\cite{Bowen11} and~\cite{Aus--soficentadd}, more recent works have considered variants of the sofic entropy which are defined in terms of sequences of measures on the spaces $\X^{V_n}$ that `converge' in a certain sense to the process $\mu$.  In particular,~\cite{Aus--soficentadd} discussed several notions of convergence for these measures, and showed how they play a role in determining when sofic entropy is additive under forming Cartesian products of systems.

The key variant of sofic entropy which appears in~\cite{Aus--soficentadd} is `model measure sofic entropy', denoted here by $\rmh_\S^\m(\mu)$.  This relies on a notion of convergence for a sequence $\mu_n \in \Pr(\X^{V_n})$ to a process $\mu$ which in that paper is called `doubly quenched' convergence.

In the present paper we consider this notion again.  However, here we use a different terminology for this mode of convergence.  We call it `local and doubly empirical' or `LDE' convergence, and denote it by $\mu_n \lde \mu$.  We propose this name change because the name `quenched convergence' could create a serious conflict with older uses of the word `quenched' in statistical physics: this is explained further in Section~\ref{sec:defs}.  Unfortunately this possible conflict did not occur to the author at the time of writing~\cite{Aus--soficentadd}.  Because of this name change, our notation $\rmh_\S^\m$ also differs from~\cite{Aus--soficentadd}: in that paper this quantity is denoted by `$\rmh_\S^{\rm{dq}}$'.

The present paper also uses the notion of local weak$^\ast$ convergence, denoted by $\mu_n \lw \mu$, which is weaker than LDE convergence.  This also played a role in~\cite{Aus--soficentadd}, but it has an older history in probability theory.

The definitions of these modes of convergence and of model-measure sofic entropy are recalled in Section~\ref{sec:defs}.

\subsection{The asymptotic equipartition property}

Our first main result is that one obtains the same value for $\rmh_\S^\m(\mu)$ if one uses only sequences $\mu_n$ which satisfy the asymptotic equipartition property (`AEP') of information theory.  The Shannon--McMillan theorem provides a similar conclusion for the finite marginals of a $\bbZ$-process, but this similarity is rather superficial: in our setting there is generally no canonical choice of a convergent sequence of measures, and we allow ourselves to change the sequence in order to secure the AEP.

To explain the result precisely, suppose first that $(\X^G,\mu,S)$ is a $G$-process with $\X$ finite.  Let $\mu_n$ be any sequence of probability measures on the spaces $\X^{V_n}$, and let $h \in [0,\infty)$.  The sequence $(\mu_n)_n$ has the \textbf{asymptotic equipartition property}, or \textbf{AEP}, with \textbf{rate} $h$ if
\begin{multline*}
\mu_n\big\{\bf{x} \in \X^{V_n}:\ e^{-h|V_n| - \eps|V_n|} < \mu_n\{\bf{x}\} < e^{-h|V_n| + \eps |V_n|}\big\} \to 1\\ \quad \hbox{as}\ n\to\infty\ \forall \eps > 0.
\end{multline*}

Some processes cannot be represented using a finite state space, so those require a generalization of this notion.  The simplest approach uses finite measurable partitions of the state space $\X$, via the following auxiliary definition.  For any measurable space $X$, probability measure $\mu$ on $X$, finite measurable partition $\P$ of $X$, and $\eps > 0$, we define
\[\cov_\eps(\mu;\P) = \min\Big\{|\P_1|:\ \P_1\subseteq \P\ \hbox{and}\ \mu\Big(\bigcup\P_1\Big) > 1-\eps\Big\}.\]
This reduces to the $\eps$-covering number $\cov_\eps(\mu)$ in case $X$ is finite and $\P$ is the partition into singletons.

Now consider a shift-invariant measure $\mu$ on $\X^G$ for an arbitrary standard measurable space $\X$, fix a finite measurable partition $\P$ of $\X$, and let $\mu_n$ be a sequence of members of $\Pr(\X^{V_n})$ as before.  For each $n$, let $\P^{V_n}$ be the partition of $\X^{V_n}$ into products of cells of $\P$. The sequence $(\mu_n)_n$ has the \textbf{AEP rel $\P$ with rate $h$} if
\begin{multline*}
\mu_n\Big(\bigcup\big\{C \in \P^{V_n}:\ e^{-h|V_n| - \eps|V_n|} < \mu_n(C) < e^{-h|V_n| + \eps |V_n|}\big\}\Big) \to 1\\ \quad \hbox{as}\ n\to\infty\ \forall \eps > 0.
\end{multline*}
This reduces to the previous definition if $\X$ is finite and $\P$ is the partition into singletons.

\vspace{7pt}

\noindent\textbf{Theorem A.}\quad \emph{Let $(\X^G,\mu,S,d)$ be a metric $G$-process, let $\P$ be a finite measurable partition of $\X$, and let $\nu_n \in \Pr(\X^{V_n})$ be a sequence such that
\begin{itemize}
 \item[i)] $\nu_n \lde \mu$, and
\item[ii)] we have
\[\frac{1}{|V_n|}\log \cov_\eps(\nu_n;\P^{V_n}) \to h\]
as $n\to\infty$ and then $\eps \downarrow 0$.
\end{itemize}}

\emph{Then there is another sequence $\mu_n \in \Pr(\X^{V_n})$ which has the same two properties and also has the AEP rel $\P$ with rate $h$.}

\vspace{7pt}

This is proved in Section~\ref{sec:A}, following two introductory sections.  In fact we prove a slightly stronger and more precise variant of this theorem: see Theorem~\ref{thm:A+} below.

\subsection{A new formula for model-measure sofic entropy}

Let $(\X^G,\mu,S,d)$ be a metric $G$-process.  Using Theorem A, we can prove a new formula for $\rmh_\S^\m(\mu)$ in terms of finite partitions of $\X$ and associated Shannon entropies.

To do this for a general space $(\X,d)$, we must now allow a whole family $\frak{P}$ of Borel partitions of $\X$, and insist that this family is compatible with $d$ in the following two ways:
\begin{itemize}
	\item[Pi)] For every $\delta > 0$ there exists $\P \in \frak{P}$ all of whose cells have diameter less than $\delta$.
	\item[Pii)] Let $\mu_e$ be the marginal of $\mu$ at the identity element $e$ of $G$. If $\P \in \frak{P}$, then every cell $C \in \P$ is a continuity set for the marginal $\mu_e$: that is,
	\[\mu_e(\partial C) = \mu\{x \in \X^G:\ x_e \in \partial C\} =0.\]
\end{itemize}
For example, if $(\X,d)$ is totally disconnected (such as an abstract Cantor set), then one may take $\frak{P}$ to consist of all partitions into subsets that are both closed and open.  For a general compact metric space $(\X,d)$, a routine argument shows the existence of families satisfying both properties (Pi) and (Pii). Indeed, for any $\delta > 0$ and $x \in \X$, all but countably many radii $r \in (0,\delta)$ must satisfy
\[\mu_e(\partial B(x,r)) = \mu_e\{y\in\X:\ d(x,y) = r\} = 0.\]
Therefore, by compactness, $\X$ has a finite cover consisting of open balls with radii less than $\delta$ and which are all continuity sets for $\mu_e$.  Now consider the Borel partition generated by all these balls and their complements. The cells of this partition are still continuity sets for $\mu_e$ and they all have diameter less than $2\delta$.

Having chosen a family $\frak{P}$ as above, we actually obtain two new formulae for $\rmh_\S^\m(\mu)$.

\vspace{7pt}

\noindent\textbf{Theorem B.}\quad \emph{Let $(\X^G,\mu,S,d)$ be a metric $G$-process and let $\frak{P}$ be any family having properties (Pi) and (Pii) above.  Then
	\begin{align*}
		\rmh_\S^\m(\mu) &= \sup\Big\{\limsup_{i\to\infty} \frac{1}{|V_{n_i}|}\log \cov_\eps(\mu_i;\P^{V_{n_i}}): \eps > 0,\ \P \in \frak{P},\\ & \qquad \qquad \qquad \qquad \qquad \qquad n_i\uparrow \infty \ \hbox{and}\ \mu_i \lde \mu\ \hbox{over}\ (\s_{n_i})_{i\geq 1}\Big\}\\
		&= \sup\Big\{\limsup_{i\to\infty} \frac{1}{|V_{n_i}|}\rmH_{\mu_i}(\P^{V_{n_i}}): \P \in \frak{P},\\ & \qquad \qquad \qquad \qquad \qquad \qquad  n_i\uparrow \infty \ \hbox{and}\ \mu_i \lde \mu\ \hbox{over}\ (\s_{n_i})_{i\geq 1}\Big\}.
	\end{align*}}

\vspace{7pt}

This is proved in Section~\ref{sec:mment}. In the proof, we first show that the two right-hand formulae agree if we fix a single partition $\P$.  This is the argument that needs Theorem A.  Then we complete the proof by showing that the first of these formulae always agrees with $\rmh_\S^\m(\mu)$.  This second step is a fairly routine exercise in regularity properties of Borel measures.  We include both formulae in the statement of Theorem B for completeness, but the second is the main new innovation.

One could formulate and prove the obvious analog of this result for the lower model-measure sofic entropy $\ul{\rmh}_\S^\m(\mu)$ (see~\cite[Section 6]{Aus--soficentadd}), but we leave the details aside.  In a different direction, a little more work should yield a generalization of Theorem B to the case of a Polish space $(\X,d)$ using similar arguments to~\cite{Hay14}, but we do not pursue this here either.

The second formula in Theorem B is useful because it makes available some of the good computational properties of Shannon entropy, such as additivity under products.  The proof of Theorem C below is an example.  Previously, a valuable special case of Theorem B was already given in~\cite[Theorem 4.1]{Bowen11}, where it was used to calculate the sofic entropies of certain actions of algebraic origin.  Another use of Shannon entropies to provide lower bounds for covering numbers in this context appeared recently in~\cite{AusBur--CPSE}.

In case $\X$ is a finite set, the best choice of $\P$ in either of the formulae in Theorem B is the partition into singletons.  From this choice we obtain the following corollary.

\vspace{7pt}

\noindent\textbf{Corollary B$'$.}\quad \emph{If $\X$ is finite, then
	\begin{align*}
		\rmh_\S^\m(\mu) &= \sup\Big\{\limsup_{i\to\infty} \frac{1}{|V_{n_i}|}\log \cov_\eps(\mu_i): \eps > 0,\\ & \qquad \qquad \qquad \qquad \qquad \qquad n_i\uparrow \infty \ \hbox{and}\ \mu_i \lde \mu\ \hbox{over}\ (\s_{n_i})_{i\geq 1}\Big\}\\
		&= \sup\Big\{\limsup_{i\to\infty} \frac{1}{|V_{n_i}|}\rmH(\mu_i):\ n_i\uparrow \infty \ \hbox{and}\ \mu_i \lde \mu\ \hbox{over}\ (\s_{n_i})_{i\geq 1}\Big\}.
	\end{align*}
	$\phantom{i}$\qed}

\vspace{7pt}

The first of these formulae already appears in~\cite[Proposition 8.1]{Aus--soficentadd}.

Section~\ref{sec:mment} also includes examples showing that hypotheses (Pi) and (Pii) on $\frak{P}$ are not superfluous in Theorem B.

\begin{rmk}
	Although a given partition $\P$ of $\X$ generates a factor of the $G$-system $(\X^G,\mu,S)$, the expressions on the right in Theorem B are not simply the sofic entropy of that factor.  Indeed, those expressions still depend on the entire metric process $(\X^G,\mu,S,d)$ because of the requirement that the measures satisfy $\mu_i \lde \mu$.  In the setting of non-amenable groups, it is known that both sofic entropy and model-measure sofic entropy can increase under factor maps, so the sofic entropy of a factor should not appear inside a supremum such as either right-hand side in Theorem B. \fin
\end{rmk}

\subsection{The entropy of certain co-induced systems}\label{subs:coind}

Now let $H$ be another sofic group with a sofic approximation $\T= (\tau_n)_{n \geq 1}$, where
\[\tau_n:H\to \rm{Sym}(W_n).\]
Then the product group $G\times H$ has a product sofic approximation $\S\times \rm{T}$ consisting of the Cartesian product maps
\[\s_n\times\tau_n:G\times H\to \rm{Sym}(V_n\times W_n):(g,h)\mapsto \s_n^g\times \tau_n^h.\]

If $(X,\mu,T)$ is a $G$-system, then the space $(X^H,\mu^{\times H})$ supports natural commuting actions of $G$ and $H$: $G$ acts diagonally through copies of $T$, and $H$ acts by coordinate left-translation.  Together these define a measure-preserving action of $G\times H$.  It is called the \textbf{co-induction} of $T$ to $G\times H$ and denoted by $\rm{CInd}_G^{G\times H}T$.  If $(X,\mu,T)$ is a $G$-process $(\X^G,\mu,S)$, then its co-induction is the $(G\times H)$-process $(\X^{G\times H},\mu^{\times H},S)$, where we use the same letter $S$ to denote the left coordinate-shift action on either of these spaces.  Co-induction can be defined more generally given any containment of two discrete groups, but we do not discuss that generality here.

In Section~\ref{sec:coind} we apply Theorem B to deduce a general formula for the behaviour of model-measure sofic entropy under co-induction of systems.

\vspace{7pt}

\noindent\textbf{Theorem C.}\quad \emph{Let $(\X^G,\mu,S,d)$ be a metric $G$ process and let $\frak{P}$ be a family of finite Borel partitions of $\X$ having properties (Pi) and (Pii).  Assume that the sofic approximation $\rm{T}$ satisfies $|W_n|\to\infty$ as $n\to\infty$. Then
\begin{multline*}
\rmh^\m_{\S\times \rm{T}}(\mu^{\times H}) = \sup\Big\{\limsup_{i\to\infty} \frac{1}{|V_{n_i}|}\rmH_{\mu_i}(\P^{V_{n_i}}):\
\P \in \frak{P},\ n_i\uparrow \infty \ \hbox{and}\\ \mu_i \lw \mu\ \hbox{over}\ (\s_{n_i})_{i\geq 1}\Big\}.
\end{multline*}}

\vspace{7pt}

This theorem gives a formula for the model-measure sofic entropy of the co-induced system $(\X^{G\times H},\mu^{\times H},S,d)$ in terms of convergent measures for the original system $(\X^G,\mu,S,d)$. It looks very similar to the second formula in Theorem B, but it has the important difference that LDE convergence has been weakened to local weak$^\ast$ convergence on the right-hand side.

The condition that $|W_n| \to \infty$ is clearly not superfluous. Without it, we could simply let $H$ be the trivial group and let each $W_n$ be a singleton, in which case $\rmh^\m_{\S\times \rm{T}}(\mu^{\times H}) = \rmh^\m_\S(\mu)$.  This is certainly not always equal to the right-hand side above, because of the difference between local weak$^\ast$ and LDE convergence. Indeed, there may exist a sequence $\mu_i \lw \mu$ even in cases where $\rmh_\S^\m(\mu) = -\infty$: see Example~\ref{ex:lw-bad} below.

Theorem C is clarified by some important special cases.  To discuss these, let us write $\t{\rmh}_\S(\mu)$ for the right-hand formula in Theorem C.

First, if $H$ is infinite, then any sofic approximation must have $|W_n|\to\infty$.  In that case,~\cite[Theorem E]{Aus--soficentadd} has already shown another entropy equality for the co-induced system: that $\rmh_{\S\times \rm{T}}(\mu^{\times H}) = \rmh_{\S\times \rm{T}}^\m(\mu^{\times H})$.  So in this case we obtain the following.

\vspace{7pt}

\noindent\textbf{Corollary C$_{\bs{1}}$.}\quad \emph{If $H$ is infinite, then
\[\rmh_{\S\times \rm{T}}(\mu^{\times H}) = \rmh^\m_{\S\times \rm{T}}(\mu^{\times H}) = \t{\rmh}_\S(\mu).\] \qed}

\vspace{7pt}

Secondly, one can let $H$ be the trivial group, and so identify $G\times H$ with $G$.  Then the maps $\tau_n$ in the sofic approximation $\rm{T}$ are also trivial, but the finite sets $W_n$ are arbitrary.  Since we assume that $|W_n|\to\infty$, this really amounts to replacing $\S$ with a new, enlarged sofic approximation $\t{\S}$ to $G$ defined by
\[\t{\s}_n^g := \s_n^g\times \rm{Id}_{W_n} \in \rm{Sym}(V_n\times W_n).\]
Equivalently, $\t{\s}_n$ is a direct sum of several copies of $\s_n$, their number tending to $\infty$ with $n$.  In this case, Theorem C tells us that this enlargement of the sofic approximation enables a new formula for model-measure sofic entropy.

\vspace{7pt}

\noindent\textbf{Corollary C$_{\bs{2}}$.}\quad \emph{In the situation described above, we have
\[\rmh^\m_{\t{\S}}(\mu) = \t{\rmh}_\S(\mu).\] \qed}

\vspace{7pt}

An immediate consequence of either Corollary C$_1$ or Corollary C$_2$ is that the quantity $\t{\rmh}_\S(\mu)$ is invariant under measure-theoretic isomorphisms of processes.  Its definition may therefore be extended unambiguously to any $G$-system $(X,\mu,T)$ by picking an arbitrary isomorphism from that system to a metric $G$-process and picking a suitable family $\frak{P}$.  We write $\t{\rmh}_\S(\mu,T)$ for the quantity obtained this way.

However, this invariant of $G$-systems is often not new.  If the metric $G$-process $(\X^G,\mu,S,d)$ is weakly mixing, then local weak$^\ast$ convergence implies LDE convergence~\cite[Corollary 5.7 and Lemma 5.15]{Aus--soficentadd}, and then $\t{\rmh}_\S(\mu)$ does coincide with the second formula in Theorem B.  This gives our third special case of Theorem C.

\vspace{7pt}

\noindent\textbf{Corollary C$_{\bs{3}}$.}\quad \emph{If $(X,\mu,T)$ is a weakly mixing $G$-system then
\[\rmh^\m_{\S\times \rm{T}}(\mu^{\times H},\rm{CInd}_G^{G\times H}T) = \rmh_\S^\m(\mu,T).\]
This is equal to $\rmh_{\S\times \rm{T}}(\mu^{\times H},\rm{CInd}_G^{G\times H}T)$ in case $H$ is infinite. \qed}

\vspace{7pt}

I do not know whether Corollary C$_3$ has a natural variant in case $\mu$ is ergodic but not weakly mixing.

In light of Corollary C$_1$, the quantity $\t{\rmh}_\S(\mu,T)$ may be viewed as a kind of `stabilized' sofic entropy.  We can obtain $\t{\rmh}_\S(\mu,T)$ by choosing any other infinite sofic group --- $\bbZ$, for example --- and then using the sofic entropy of the co-induced system
\[\big(X^\bbZ,\mu^{\times \bbZ},\rm{CInd}_G^{G\times \bbZ}T\big).\]
Alternatively, by Corollary C$_2$, we may retain the group $G$ but enlarge the sofic approximation to $\t{\S}$. In particular, either approach gives a new formula for $\rmh_\S^\m(\mu,T)$ in case $(X,\mu,T)$ is weakly mixing, by Corollary C$_3$. Much like the relation to power-stabilized sofic entropy established in~\cite[Section 8]{Aus--soficentadd}, this gives a meaning to model-measure sofic entropy which does not refer to any sequences of measures on model spaces.  It is restricted to weakly mixing $G$-systems, but avoids another problem with the results of~\cite[Section 8]{Aus--soficentadd}: those are presently known only for systems with finite generating partitions.

\subsection{Next steps}

A few remaining open problems are scattered below.  However, another problem which deserves mention here is to generalize any of the results above to the relative setting: that is, so that they apply to a notion of relative sofic entropy for a factor map $(X,\mu,T) \to (Y,\nu,S)$ between two $G$-systems.

\section{Background on LDE convergence and model-measure sofic entropy}\label{sec:defs}

Aside from our new terminology for LDE convergence, our formalism and notation essentially agree with those of~\cite{Aus--soficentadd}.  This section simply recalls the main definitions.  More discussion can be found in that paper.

Let $G$ be a countable sofic group and let $\S = (\s_n)_n$ be a sofic approximation to it, where $\s_n:G\to \rm{Sym}(V_n)$.  Let $(\X^G,\mu,S,d)$ be a metric $G$-process.  We endow each of the spaces $\X^{V_n}$ with the compact product topology, and more specifically with the \textbf{Hamming average} $d^{(V_n)}$ of the metric $d$, defined by
\[d^{(V_n)}(\bf{x},\bf{y}) := \frac{1}{|V_n|}\sum_{v \in V_n}d(x_v,y_v).\]

Given $\bf{x} \in \X^{V_n}$ and $v \in V_n$, the \textbf{pullback name of $\bf{x}$ at $v$} is the element
\[\Pi^{\s_n}_v(\bf{x}) = (x_{\s_n^g(v)})_{g\in G} \in \X^G.\]
The \textbf{empirical distribution} of $\bf{x}$ is
\[P^{\s_n}_\bf{x} = \frac{1}{|V_n|}\sum_{v \in V_v}\delta_{\Pi^{\s_n}_v(\bf{x})} \in \Pr(\X^G).\]

If $U$ is a weak$^\ast$ neighbourhood of $\mu$ in $\Pr(\X^G)$, then the \textbf{$U$-good models} for $\mu$ over $\s_n$ are the elements of the set
\[\O(U,\s_n) = \big\{\bf{x} \in \X^{V_n}:\ P^{\s_n}_\bf{x} \in U\big\}.\]

The \textbf{sofic entropy of $\mu$ rel $\S$} is the quantity
\[\rmh_\S(\mu) := \sup_{\delta > 0}\inf_U\limsup_{n\to\infty}\frac{1}{|V_n|}\log \cov_\delta\big(\O(U,\s_n),d^{(V_n)}\big),\]
where the infimum is over all weak$^\ast$ neighbourhoods $U$ of $\mu$, and
\[\cov_\delta\big(\O(U,\s_n),d^{(V_n)}\big)\]
denotes the least number of $\delta$-balls in the metric $d^{(V_n)}$ needed to cover $\O(U,\s_n)$.  See~\cite{Aus--soficentadd} for a comparison of this with older definitions of sofic entropy.

Now suppose in addition that $\mu_n \in \Pr(\X^{V_n})$ for each $n$.

\begin{dfn}
The sequence $(\mu_n)_n$ \textbf{locally weak$^\ast$ converges to $\mu$ over $\S$} if
\[\big|\big\{v \in V_n:\ (\Pi^{\s_n}_v)_\ast\mu_n \in U\big\}\big| = (1 - o(1))|V_n|\]
for every weak$^\ast$ neighbourhood $U$ of $\mu$.  This is denoted by $\mu_n \lw \mu$.
\end{dfn}

The next definition was the first main innovation of~\cite{Aus--soficentadd}, but here it has its new name.

\begin{dfn}
The sequence $(\mu_n)_n$ \textbf{locally and empirically converges} (or `\textbf{LE-converges}') {to $\mu$ over $\S$} if
\begin{itemize}
\item[i)] $\mu_n \lw \mu$, and also
\item[ii)] $\mu_n(\O(U,\s_n))\to 1$ for every weak$^\ast$ neighbourhood $U$ of $\mu$.
\end{itemize}
We denote this by $\mu_n \stackrel{\rm{le}}{\to} \mu$.

The sequence $(\mu_n)_n$ \textbf{locally and doubly empirically converges} (or `\textbf{LDE converges}') {to $\mu$ over $\S$} if $\mu_n \times \mu_n \stackrel{\rm{le}}{\to} \mu \times \mu$ over $\S$.  We denote this by $\mu_n \lde \mu$.
\end{dfn}

Section 5 of~\cite{Aus--soficentadd} is devoted to these notions of convergence. In that paper, LE-convergence is called `quenched convergence', and LDE convergence is called `doubly-quenched convergence'.  This usage was motivated by the use of the word `quenched' in statistical physics: it indicates a property that holds with high probability as some thermodynamic or other limit is approached, rather than just on average.  However, I now realize that my previous terminology could create a serious conflict. In case $G$ is a free group, it is common and natural to produce sofic approximations to it randomly (see, for instance,~\cite{Bowen10c}).  In this case, if we are to use `quenched' as in statistical physics, it should be applied to the random choice \emph{of sofic approximation}, not to the choice of model given a fixed sofic approximation.  This is the why I have switched to the new names above for these modes of convergence for the measures $\mu_n$.

Finally, the \textbf{model-measure sofic entropy of $(\X^G,\mu,S,d)$ rel $\S$} is the quantity
\begin{multline*}
	\rmh_\S^\m(\mu) = \sup\Big\{\sup_{\delta,\eps > 0} \limsup_{i\to\infty} \frac{1}{|V_{n_i}|}\log \cov_{\eps,\delta}(\mu_i,d^{(V_{n_i})}):\\ n_i\uparrow \infty \ \hbox{and}\ \mu_i \lde \mu\ \hbox{over}\ (\s_{n_i})_{i\geq 1}\Big\},
\end{multline*}
where $\cov_{\eps,\delta}(\mu_i,d^{(V_{n_i})})$ denotes the least number of $\delta$-balls in the metric $d^{(V_{n_i})}$ needed to cover more than $1-\eps$ of the measure $\mu_i$.

Model-measure sofic entropy is defined and studied in~\cite[Section 6]{Aus--soficentadd} (but beware that it also has a different notation: it is denoted by `$\rmh_\S^{\rm{dq}}(\mu)$' in~\cite{Aus--soficentadd}).  According to~\cite[Theorem 6.4]{Aus--soficentadd}, it is an invariant of the process under measure-theoretic isomorphism, so in particular it does not depend on the choice of $d$, and can be defined unambiguously for any $G$-system.

\section{A stability result for LDE convergence}

The new sequence of measures $\mu_n$ given by Theorem A is obtained by conditioning: this is made explicit in the refined statement of Theorem~\ref{thm:A+} below.  It is crucial for the proof that conditioning on sets which are not too small cannot disrupt LDE convergence.  We state and prove this as a separate theorem since it has independent interest and may find other applications elsewhere.

Fix a metric $G$-process $(\X^G,\mu,S,d)$.

\begin{thm}[Stability of LDE convergence under conditioning]\label{thm:stability}
	Suppose that ${\mu_n \lde \mu}$, and let $A_n \subseteq \X^{V_n}$ be subsets which all satisfy $\mu_n(A_n) \geq a$ for some $a > 0$. Then also $\mu_n(\,\cdot\,|\,A_n) \lde \mu$.
\end{thm}

The proof of this rests on the following elementary lemma.

\begin{lem}\label{lem:averaging}
Let $X$ be a compact metrizable space and let $\Pr(X)$ be the space of all Borel probability measures on it, endowed with the weak$^\ast$ topology.  Let $\theta$ be a probability distribution on $\Pr(X)$, and let $\mu$ be its barycentre, meaning that
\begin{equation}\label{eq:bary}
\mu(A) = \int \nu(A)\,\theta(d\nu) \quad \forall \ \rm{Borel}\ A \subseteq X. 
\end{equation}
Suppose in addition that
\begin{equation}\label{eq:squares-bary}
\mu\times \mu = \int \nu\times \nu\ \theta(d\nu).
\end{equation}
Then $\theta$ is the point mass at $\mu$.
\end{lem}

\begin{proof}
If $A \subseteq X$ is Borel, then~(\ref{eq:squares-bary}) implies that
\[\mu(A)^2 = \int \nu(A)^2\,\theta(d\nu).\]
In combination with~(\ref{eq:bary}), this is possible only if $\nu(A) = \mu(A)$ for $\theta$-a.e. $\nu$.  Letting $A$ run through a countable sequence which generates the Borel $\s$-algebra of $X$, it follows that $\theta(\{\mu\}) = 1$.
\end{proof}

\begin{proof}[Proof of Theorem~\ref{thm:stability}]
Let $\nu_n := \mu_n(\,\cdot\,|\,A_n)$ for each $n$.

\vspace{7pt}
	
\emph{Part 1: support on good models}.\quad Our assumptions include that ${\mu_n \stackrel{\rm{le}}{\to} \mu}$, so for any weak$^\ast$ neighbourhood $U$ of $\mu$ in $\Pr(\X^G)$ we have
\begin{align*}
\nu_n(\O(U,\s_n)) &= \frac{\mu_n(\O(U,\s_n)\cap A_n)}{\mu_n(A_n)}\\ &\geq \frac{\mu_n(A_n) - \mu_n(\X^{V_n}\setminus \O(U,\s_n))}{\mu_n(A_n)}\\
&\geq 1 - \frac{\mu_n(\X^{V_n}\setminus \O(U,\s_n))}{a}\\
&\to 1 \quad \hbox{as}\ n\to\infty.
\end{align*}
So $\nu_n$ is asymptotically supported on good models for $\mu$.

Using the full assumption that $\mu_n \lde \mu$, we apply the same reasoning to the measures $\mu_n\times \mu_n$ and sets $A_n\times A_n$.  It follows that $\nu_n\times \nu_n$ is asymptotically supported on good models for $\mu\times \mu$.

\vspace{7pt}

\emph{Part 2: local weak$^\ast$ convergence}.\quad Next we show that $\nu_n \lws \mu$.  For each $v \in V_n$, let
\[\nu_{n,v} := (\Pi^{\s_n}_v)_\ast\nu_n \in \Pr(\X^G).\]
Let $\theta_n$ be the uniform distribution on this collection of elements of $\Pr(\X^G)$: that is,
\[\theta_n := \frac{1}{|V_n|}\sum_{v \in V_n}\delta_{\nu_{n,v}} \in \Pr(\Pr(\X^G)).\]

The desired local weak$^\ast$ convergence asserts that
\[\theta_n(U) = \frac{|\{v \in V_n:\ \nu_{n,v} \in U\}|}{|V_n|} \to 1 \quad \hbox{as}\ n\to\infty\]
for any weak$^\ast$ neighbourhood $U$ of $\mu$.  This is equivalent to the convergence
\[\theta_n \stackrel{\rm{weak}^\ast}{\to} \delta_\mu,\]
where this refers to convergence in the weak$^\ast$ topology of $\Pr(\Pr(\X^G))$.

That topology is compact, by Alaoglu's theorem, so it suffices to prove that $\delta_\mu$ is the only subsequential weak$^\ast$ limit of the sequence $(\theta_n)_n$.  Let $\theta = \lim_{i\to\infty} \theta_{n_i}$ be any such subsequential limit.  Since the formation of product measures is continuous for the weak$^\ast$ topology, it follows that
\[\int \nu\times \nu\ \theta(d\nu) = \lim_{i\to\infty} \int \nu\times \nu\ \theta_{n_i}(d\nu) = \lim_{i\to\infty} \frac{1}{|V_{n_i}|}\sum_{v \in V_{n_i}}\nu_{n_i,v}\times \nu_{n_i,v}.\]
The average inside this last limit may be re-arranged as
\begin{align*}
\frac{1}{|V_{n_i}|}\sum_{v \in V_{n_i}}\nu_{n_i,v}\times \nu_{n_i,v} &= \frac{1}{|V_{n_i}|}\sum_{v \in V_{n_i}}(\Pi^{\s_{n_i}}_v)_\ast(\nu_{n_i}\times \nu_{n_i}) \\ &= \iint P^{\s_n}_{(\bf{x},\bf{x}')}\ \nu_{n_i}(d\bf{x})\,\nu_{n_i}(d\bf{x}').
\end{align*}
However, Part 1 has shown that this last integral converges to $\mu\times \mu$.  Therefore $\theta$ must satisfy
\[\int \nu\times \nu\ \theta(d\nu) = \mu\times \mu,\]
and so $\theta = \delta_\mu$, by Lemma~\ref{lem:averaging}.

This shows that $\nu_n \lws \mu$, and for local weak$^\ast$ convergence this implies immediately that $\nu_n\times \nu_n \lws \mu\times \mu$.
\end{proof}

\section{Obtaining sequences with the AEP}\label{sec:A}

This section proves Theorem A and some related results.  In fact we formulate and prove a version which gives a slightly stronger conclusion and which describes how the new measures are obtained.

For the statement, let $(\X^G,\mu,S)$, $\P$, and $(\nu_n)_n$ be as in the statement of Theorem A, satisfying the properties (i) and (ii) listed there.  We say that a sequence $\mu_n \in \Pr(\X^{V_n})$ has the \textbf{strong AEP rel $\P$ with rate $h$} if for every $\eps > 0$ we have
\[\mu_n\Big(\bigcup\big\{C \in \P^{V_n}:\ e^{-h|V_n| - \eps|V_n|} < \mu_n(C) < e^{-h|V_n| + \eps |V_n|}\big\}\Big) = 1\]
for all sufficiently large $n$.

\begin{thm}[Stronger version of Theorem A]\label{thm:A+}
	There are measurable sets $A_n\subseteq \X^{V_n}$ with $\nu_n(A_n) \geq e^{-o(n)}$ such that the sequence of conditional measures
	\[\mu_n := \nu_n(\,\cdot\,|\,A_n)\]
	has properties (i) and (ii) and also has the strong AEP rel $\P$ with rate $h$.
\end{thm}

We present the proof following a couple of preparatory lemmas.  The first justifies the use of diagonal arguments for LDE-convergent sequences of measures.

\begin{lem}[Passing to a diagonal subsequence]\label{lem:diag}
Suppose $\mu_{k,n} \in \Pr(\X^{V_n})$ for every $k,n \in \bbN$, and that $\mu_{k,n} \lde \mu$ as $n\to\infty$ for every fixed $k$.  Then also $\mu_{k_n,n} \lde \mu$ whenever the integer sequence $k_1 \leq k_2 \leq \dots$ tends to $\infty$ sufficiently slowly.
\end{lem}

\begin{proof}
Let $U_1 \supseteq U_2 \supseteq \dots$ be a basis of weak$^\ast$ neighbourhoods around $\mu\times \mu$.  By assumption, for each $k$ there exists $N_k$ such that
\[\big|\big\{v\in V_n:\ (\Pi_v^{\s_n})_\ast(\mu_{k,n} \times \mu_{k,n}) \in U_k\big\}\big| \geq (1-2^{-k})|V_n| \qquad \forall n\geq N_k\]
and
\[(\mu_{k,n} \times \mu_{k,n})\big(\O(U_k,\s_n)\big) \geq 1 - 2^{-k} \qquad \forall n\geq N_k.\]
Provided $(k_n)_n$ tends to $\infty$ so slowly that $n\geq N_{k_n}$ for all sufficiently large $n$, these imply the LDE convergence of $\mu_{k_n,n}$.
\end{proof}

Secondly, we need to know that having the AEP with a given rate actually controls the Shannon entropy and covering numbers of a sequence of measures.  The following simple facts can be traced to arguments in~\cite{Shannon48}.  The proof is standard, but we include it for completeness.

\begin{lem}\label{lem:AEP-and-others}
	If $(\mu_n)_n$ satisfies the AEP rel $\P$ with rate $h$, then also
	\[\frac{1}{|V_n|}\log\cov_\eps(\mu_n;\P^{V_n}) \to h \quad \hbox{as}\ n\to\infty\]
	for all $\eps \in (0,1)$, and 
	\[\frac{1}{|V_n|}\rmH_{\mu_n}(\P^{V_n}) \to h \quad \hbox{as}\ n\to\infty.\]	
\end{lem}

\begin{proof}
For each $n$, let $\eps_n > 0$, let
\[\Q_n := \{C \in \P^{V_n}:\ e^{-h|V_n| - \eps_n|V_n|} < \mu_n(C) < e^{-h|V_n| + \eps_n|V_n|}\},\]
and let $A_n := \bigcup \Q_n \subseteq \X^{V_n}$. By the AEP, if we choose $\eps_n$ tending to zero sufficiently slowly, then these sets satisfy $\mu_n (A_n) \to 1$.  We now assume this.

Suppose that $0 < \eps < 1$. Then $\mu_n(A_n) > 1 - \eps$ for all sufficiently large $n$. Since the measures $\mu_n(C)$ for $C \in \Q_n$ must sum to at most $1$, we have
\[|\Q_n| \leq e^{h|V_n| + \eps_n|V_n|} = e^{h|V_n| + o(|V_n|)},\]
and so $\cov_\eps(\mu_n;\P^{V_n}) \leq e^{h|V_n| + o(|V_n|)}$.  On the other hand, consider any $\P_1 \subseteq \P^{V_n}$ satisfying $\mu_n(\bigcup \P_1) > 1-\eps$, and choose some $\eps' \in (\eps,1)$.  If $n$ is sufficiently large, then we have
\[\mu_n\Big(\bigcup(\cal{Q}_n\cap \P_1)\Big) > 1 - \eps' > 0.\]
Since any cell in $\cal{Q}_n\cap \P_1$ has measure at most $e^{-h|V_n| + \eps_n|V_n|}$, this implies that
\[|\P_1| \geq |\cal{Q}_n\cap \P_1| > (1-\eps')e^{h|V_n| - \eps_n|V_n|} = e^{h|V_n| - o(|V_n|)}\]
once $n$ is sufficiently large. So in fact $\cov_\eps(\mu_n;\P^{V_n}) = e^{h|V_n| + o(|V_n|)}$.

Lastly, we may estimate $\rmH_{\mu_n}(\P^{V_n})$ by first conditioning on the coarser partition into $A_n$ and $\X^{V_n}\setminus A_n$:
\begin{multline*}
\rmH_{\mu_n}(\P^{V_n}) = \rmH_{\mu_n}(A_n,\X^{V_n}\setminus A_n) + \mu_n(A_n)\rmH_{\mu_n(\,\cdot\,|A_n)}(\P^{V_n}) \\ + (1 - \mu_n(A_n))\rmH_{\mu_n(\,\cdot\,|\,\X^{V_n}\setminus A_n)}(\P^{V_n}).
\end{multline*}
All three right-hand terms are non-negative.  The first is at most $\log 2$, and the last is at most
\[(1 - \mu_n(A_n))\cdot \log|\P^{V_n}| = (1 - \mu_n(A_n))\cdot \log|\P|\cdot |V_n| = o(|V_n|),\]
since $\mu_n(A_n) \to 1$.  Finally, since the values $\mu_n(C)$ for $C \in \Q_n$ are all equal to $e^{-h|V_n| + o(|V_n|)}$, a simple calculation gives
\[\rmH_{\mu_n(\,\cdot\,|A_n)}(\P^{V_n}) = h|V_n| + o(|V_n|).\]
Therefore the same asymptotic holds for $\rmH_{\mu_n}(\P^{V_n})$.
\end{proof}

\begin{proof}[Proof of Theorem~\ref{thm:A+}]
Let
\[\Q_{a,n} := \{C \in \P^{V_n}:\ \nu_n(C) > e^{-a|V_n|}\}.\]
for each $n\in\bbN$ and any $a \geq 0$.  Clearly $|\Q_{a,n}| < e^{a|V_n|}$, and so for any $k \in \bbN$ we must have
\[\limsup_{n\to\infty}\nu_n\Big(\bigcup \Q_{h-1/k,n}\Big) < 1,\]
for otherwise the subfamily $\Q_{h-1/k,n} \subseteq \P^{V_n}$ would violate assumption (ii) in Theorem A.  On the other hand, in order to compute $\cov_\eps(\nu_n;\P^{V_n})$, it is most efficient to cover $\nu_n$ using cells of $\P^{V_n}$ in order of decreasing $\nu_n$-measure.  So condition (ii) in Theorem A also implies that
\[\nu_n\Big(\bigcup \Q_{h+1/k,n}\Big) \to 1\quad \hbox{as}\ n\to\infty\]
for any $k\in\bbN$.

For each $k \in \bbN$, let
\[A_{k,n} := \bigcup \Q_{h+1/k,n}\Big\backslash \bigcup\Q_{h-1/k,n}\]
and let $\mu_{k,n} := \nu_n(\,\cdot\,|A_{k,n})$.  The estimates above show that $\nu_n(A_{k,n})$ is bounded away from $0$ as $n\to\infty$ for each $k$, so Theorem~\ref{thm:stability} shows that we still have $\mu_{k,n} \lde \mu$ for each $k$.  Given this, Lemma~\ref{lem:diag} provides a sequence $k_1 \leq k_2 \leq \dots$ growing slowly to $\infty$ so that $\mu_{k_n,n}\lde \mu$.  Provided that sequence grows slowly enough, we may also assume that $\nu_n(A_{k_n,n}) \geq e^{-o(n)}$.

Letting $A_n := A_{k_n,n}$ and $\mu_n := \mu_{k_n,n}$, these are of the form asserted in Theorem~\ref{thm:A+}. Let us show that they have all the desired properties. We have already deduced conclusion (i).  For the rest, observe that any $C \in \Q_{h+1/k_n,n}\setminus \Q_{h-1/k_n,n}$ satisfies
\[\frac{1}{\nu_n(A_{k_n,n})}e^{-(h+1/k_n)|V_n|} < \mu_n(C) \leq \frac{1}{\nu_n(A_{k_n,n})}e^{-(h-1/k_n)|V_n|},\]
and that $\mu_n(C) = 0$ for any other $C \in \P^{V_n}$.  Since $\nu_n(A_{k_n,n}) \geq e^{-o(n)}$ and $1/k_n \to 0$, this shows the strong AEP rel $\P$ with rate $h$.  By Lemma~\ref{lem:AEP-and-others}, this also implies conclusion (ii) of the theorem.
\end{proof}

Theorem~\ref{thm:A+} tells us how the new measures $\mu_n$ are obtained from $\nu_n$ in Theorem A. This procedure has further consequences of its own.  As an illustration, let us show that we can actually obtain the AEP relative to a whole family of partitions simultaneously.  This requires one more preparatory lemma.

\begin{lem}\label{lem:big-enough}
	Let $\P$ be a finite measurable partition of $\X$ and let $h \in [0,\infty)$. Let $\nu_n \in \Pr(\X^{V_n})$ for each $n$, and let $B_n \subseteq A_n \subseteq \X^{V_n}$ be such that
	\begin{itemize}
		\item[i)] $\nu_n(A_n) > 0$ for all $n$ and the conditional measures
		\[\nu_n(\,\cdot\,|\,A_n)\]
		have the strong AEP rel $\P$ with rate $h$, and
		\item[ii)] $\nu_n(B_n) \geq e^{-\k_n}\nu_n(A_n)$ for some sequence $\k_n = o(n)$ as $n\to\infty$.
	\end{itemize}
Then the conditional measures $\nu_n(\,\cdot\,|\,B_n)$ have the AEP (not necessarily strong) rel $\P$ with rate $h$.
\end{lem}
	
\begin{proof}
	Let $\eps > 0$, and let
	\[\cal{Q}_n := \big\{C \in \P^{V_n}:\ e^{-h|V_n| - \eps|V_n|} < \nu_n(C\,|\,A_n) < e^{-h|V_n| + \eps|V_n|} \big\}.\]
	Assumption (i) gives that
	\[\nu_n\Big(\bigcup \Q_n\,\Big|\,A_n\Big) = 1\]
	for all sufficiently large $n$, so the same holds with $B_n$ in place of $A_n$ (this is the point at which we need to assume the \emph{strong} AEP).  Also, since the values $\nu(C\,|\,A_n)$ for $C \in \Q_n$ must sum to at most $1$, we have $|\Q_n| < e^{h|V_n| + \eps|V_n|}$ for every $n$.

If $C \in \Q_n$, then
	\[\nu_n(C\,|\,B_n) = \frac{\nu_n(C\cap B_n)}{\nu_n(B_n)}\leq e^{\k_n}\nu_n(C\,|\,A_n)< e^{-h|V_n| + \eps|V_n| + \k_n}.\]
	This is less than $e^{-h|V_n| + 2\eps|V_n|}$ for all sufficiently large $n$. Now let
	\[\Q'_n := \big\{C \in \Q_n:\ \nu_n(C\,|\,B_n) \geq e^{-h|V_n| - 2\eps|V_n|}\big\}.\]
Then the definition of $\Q'_n$ gives
	\[\nu_n\Big(\bigcup(\Q_n\setminus \Q_n')\,\Big|\,B_n\Big) \leq |\Q_n|e^{-h|V_n| - 2\eps|V_n|} < e^{-\eps|V_n|} \to 0\]
	as $n\to\infty$.  Since $\eps$ was arbitrary, the family $\Q'_n$ witnesses the AEP of $\nu_n(\,\cdot\,|\,B_n)$ rel $\P$ with rate $h$.
\end{proof}
	
\begin{cor}
Let $(\X^G,\mu,S,d)$ be a metric $G$-process and let $(\P_k)_k$ be a sequence of finite measurable partitions of $\X$.  Let $h_k \in [0,\infty)$ for each $k$, and let $\nu_n \in \Pr(\X^{V_n})$ be a sequence such that
\begin{itemize}
 \item[i)] $\nu_n \lde \mu$, and
\item[ii)] for every $k$ we have
\[\frac{1}{|V_n|}\log \cov_\eps(\nu_n;\P_k^{V_n}) \to h_k\]
as $n\to\infty$ and then $\eps \downarrow 0$.
\end{itemize}

Then there is another sequence $\mu_n \in \Pr(\X^{V_n})$ which has the same two properties and which also has the AEP rel $\P_k$ with rate $h_k$ for every $k$.
\end{cor}

\begin{proof}
First, Theorem~\ref{thm:A+} gives a sequence of subsets $A_{1,n} \subseteq \X^{V_n}$ with $\nu_n(A_{1,n}) \geq e^{-o(n)}$ and such that the conditional measures
\[\mu_{1,n} := \nu_n(\,\cdot\,|\,A_{1,n})\]
satisfy properties (i) and (ii) and also the strong AEP rel $\P_1$ with rate $h_1$.

Now we may apply Theorem~\ref{thm:A+} again, this time to the sequence $\mu_{1,n}$, to find subsets $A_{2,n} \subseteq \X^{V_n}$ with $\mu_{1,n}(A_{2,n}) \geq e^{-o(n)}$ and such that the conditional measures
\[\mu_{2,n} := \mu_{1,n}(\,\cdot\,|\,A_{2,n})\]
satisfy properties (i) and (ii) and also the strong AEP rel $\P_2$ with rate $h_2$.  By intersecting with $A_{1,n}$ if necessary, we may clearly also choose $A_{2,n}\subseteq A_{1,n}$.

Continuing in this way, a recursion on $k$ produces a doubly-indexed array $(A_{k,n})_{k,n}$ such that the following hold:
\begin{itemize}
	\item[a)] For each $n$, we have
	\[\X^{V_n}\supseteq A_{1,n}\supseteq A_{2,n} \supseteq A_{3,n}\supseteq \cdots,\]
all these sets have positive measure according to $\nu_n$, and
\[\nu_n(A_{k+1,n}) \geq e^{-o(n)}\nu_n(A_{k,n}) \quad \hbox{as}\ n\to\infty\]
for each $k$.
\item[b)] For each $k$, the sequence of measures $\nu_n(\,\cdot\,|\,A_{k,n})$ satisfies properties (i) and (ii) and also the strong AEP rel $\P_k$ with rate $h_k$.
\end{itemize}

Using these properties and Lemma~\ref{lem:diag}, it follows that if we choose $k_1 \leq k_2 \leq \cdots$ sufficiently slowly, then the sequence of sets $B_n := A_{k_n,n}$ satisfies
\begin{equation}\label{eq:Bs-big}
\nu_n(B_n) \geq e^{-o(n)}\nu_n(A_{k,n}) \quad \hbox{as}\ n\to\infty
\end{equation}
for each fixed $k$, and also the sequence
\[\mu_n := \nu_n(\,\cdot\,|\,B_n)\]
LDE converges to $\mu$. Now property (b) above, the bound~(\ref{eq:Bs-big}), and Lemma~\ref{lem:big-enough} imply that $(\mu_n)_{n\geq 1}$ still satisfies the AEP rel $\P_k$ with rate $h_k$ for every $k$.  Finally, property (ii) still holds for $\mu_n$ by Lemma~\ref{lem:AEP-and-others}.
\end{proof}

Before leaving this section, let us consider how the results above change if the assumption of LDE convergence is weakened to local weak$^\ast$ convergence.

This makes no difference for $G$-systems that are weakly mixing, for then the two modes of convergence coincide~\cite[Corollary 5.7 and Lemma 5.15]{Aus--soficentadd}. However, if $(\X^G,\mu,S)$ is not even ergodic, then Theorem A has no analog for local weak$^\ast$ convergence.

\begin{ex}\label{ex:lw-bad}
	Let $G$ be a sofic group with a finite generating set $S = S^{-1}$, and let $\S = (\s_n)_{n\geq 1}$ be a sofic approximation to $G$ with the property that the graphs
	\[G_n = \{(v,\s_n^s(v)):\ v \in V_n,\ s \in S\} \cup \{(\s_n^s(v),v):\ v\in V_n,\ s\in S\}\]
are an expander sequence.  (This holds, for instance, for a `typical' sofic approximation to a non-amenable free group, or for a sequence of finite quotients of a group with Kazhdan's property (T): see~\cite[Example 5.9]{Aus--soficentadd}.)  This assumption has the consequence that, for any sequence of elements $\bf{x}_n \in \X^{V_n}$, any weak$^\ast$ limit of the empirical measures $P^{\s_n}_{\bf{x}_n}$ in $\Pr^S(\X^G)$ must be ergodic.
	
	Let $\X$ be a finite set with at least two elements, and let $p,q \in \Pr(\X)$ be two probability distributions with $\rmH(p) > \rmH(q)$.  Now let
	\[\mu := \frac{1}{2}p^{\times G} + \frac{1}{2}q^{\times G}.\]
	This is a mixture of two ergodic measures.  Similarly, let
	\[\mu_n := \frac{1}{2}p^{\times V_n} + \frac{1}{2}q^{\times V_n}.\]
	It is easy to show that $\mu_n \lw \mu$.  However, this is not empirical convergence.  Indeed, once $n$ is large, an element $\bf{x} \in \X^{V_n}$ drawn at random with distribution $\mu_n$ has empirical distribution $P^{\s_n}_\bf{x}$ which is probably either close to $p^{\times G}$ or close to $q^{\times G}$, each with probability about $1/2$, but in neither case is $\bf{x}$ a good model for the whole of the measure $\mu$.
	
	A simple calculation gives that
	\begin{equation}\label{eq:cov-H-p}
\log\rm{cov}_\eps(\mu_n) = \rmH(p)\cdot |V_n| + o(|V_n|)
\end{equation}
	for any $\eps \in (0,1/2)$.  This is because, in order to cover at least $1 - \eps > 1/2$ of the measure $\mu_n$, one must in particular cover a positive fraction of the good models for $p^{\times V_n}$.  All such good models have probability
	\[\exp\big(-\rmH(p)\cdot |V_n| - o(|V_n|)\big)\]
	according to $\mu_n$, so such a covering requires
	\[\exp\big(\rmH(p)\cdot |V_n| + o(|V_n|)\big)\]
	points of $\X^{V_n}$.
	
	However, now suppose that $\nu_n \in \Pr(\X^{V_n})$ is a sequence which satisfies the AEP and has $\nu_n \lw \mu$.  By a simple adaptation of the proof of~\cite[Corollary 5.7]{Aus--soficentadd}, it follows that, once $n$ is large, $\nu_n$ is mostly supported on good models for either $p^{\times G}$ or $q^{\times G}$, each with probability about $1/2$, just as we saw above for the measure $\mu_n$.  But there are only at most
	\[\exp\big(\rmH(q)\cdot|V_n| + o(|V_n|)\big)\]
	good models for $q^{\times G}$, so if $(\nu_n)_n$ satisfies the AEP, then it must do so with rate at most $\rmH(q)$.  Since this is strictly less than $\rmH(p)$, we cannot match the exponential growth rate of the covering numbers in~\eqref{eq:cov-H-p}. \fin
\end{ex}

This leaves us with the following.

\begin{ques}
	Does Theorem A still hold if `$\lde$' is weakened to `$\stackrel{\rm{le}}{\to}$'? \fin
\end{ques}

\section{Alternative formulae for model-measure sofic entropy}\label{sec:mment}

\subsection{Proof of the new formulae}

This section proves Theorem B.  First we prove a preliminary result that relates the two right-hand formulae in that theorem to each other.  This is the point where we need Theorem A to provide LDE-convergent sequences of measures which have the AEP.

\begin{prop}\label{prop:mment}
Let $(\X^G,\mu,S,d)$ be a metric $G$-process and let $\P$ be a finite measurable partition of $\X$.  Then
\begin{align*}
&\sup\Big\{\limsup_{i\to\infty} \frac{1}{|V_{n_i}|}\log \cov_\eps(\mu_i;\P^{V_{n_i}}): \eps > 0,\\ &
\qquad \qquad \qquad \qquad \qquad \qquad   n_i\uparrow \infty \ \hbox{and}\ \mu_i \lde \mu\ \hbox{over}\ (\s_{n_i})_{i\geq 1}\Big\}\\
&= \sup\Big\{\limsup_{i\to\infty} \frac{1}{|V_{n_i}|}\rmH_{\mu_i}(\P^{V_{n_i}}):\ n_i\uparrow \infty \ \hbox{and}\ \mu_i \lde \mu\ \hbox{over}\ (\s_{n_i})_{i\geq 1}\Big\}.
\end{align*}
\end{prop}

\begin{proof}
	Let $h_1$ and $h_2$ be the left- and right-hand sides here, respectively.
		
	\vspace{7pt}
	
	\emph{Step 1.}\quad Fix any sequences $(\s_{n_i})_{i\geq 1}$ and $(\mu_i)_{i \geq 1}$. For any $i\geq 1$ and $\Q \subseteq \P^{V_{n_i}}$, we may let $A = \bigcup \Q$ and then condition on the partition $\{A,\X^{V_{n_i}}\setminus A\}$ to obtain
	\begin{align*}
	\rmH_{\mu_i}(\P^{V_{n_i}}) &= \rmH_{\mu_i}(A,\X^{V_{n_i}}\setminus A) + \mu_i(A)\rmH_{\mu_i(\,\cdot\,|A)}(\P^{V_{n_i}}) \\ & \qquad \qquad \qquad \qquad + (1-\mu_i(A))\rmH_{\mu_i(\,\cdot\,|\X^{V_{n_i}}\setminus A)}(\P^{V_{n_i}})\\ &\leq \log 2 + \mu_i(A)\log|\Q| + (1-\mu_i(A))|V_{n_i}|\log|\P|.
	\end{align*}
	Optimizing over collections $\Q$ for which $\mu_i(A) > 1 - \eps$, this gives
	\[\rmH_{\mu_i}(\P^{V_{n_i}}) \leq \log 2 + \log\cov_\eps(\mu_i;\P^{V_n}) + \eps|V_{n_i}|\log|\P|.\]
	Since $\eps$ may be made arbitrarily small, this gives $h_1 \geq h_2$.
	
	\vspace{7pt}
	
	\emph{Step 2.}\quad Fix any sequences $(\s_{n_i})_{i\geq 1}$ and $(\mu_i)_{i\geq 1}$ and also $\eps > 0$.  By passing to a further subsequence if necessary, we may assume that the quantities
	\[\frac{1}{|V_{n_i}|}\log \cov_\eps(\mu_i;\P^{V_{n_i}})\]
	actually converge to some $h\geq 0$ as $i\to\infty$, where this $h$ is the limit supremum for the original sequences.
	
	Now Theorem A provides a new sequence of measures $\nu_i$ which still satisfy $\nu_i \lde \mu$ over $(\s_{n_i})_{i\geq 1}$ and which also satisfy the AEP rel $\P$ with rate $h$.  Then Lemma~\ref{lem:AEP-and-others} gives
	\[\frac{1}{|V_{n_i}|}\rmH_{\nu_i}(\P^{V_{n_i}})\to h,\]
	so this shows that $h_1 \leq h_2$.	
\end{proof}

\begin{proof}[Proof of Theorem B]
Proposition~\ref{prop:mment} gives the equality of the two right-hand expressions.  Calling their common value $h$, we need only show that it equals $\rmh_\S^\m(\mu)$.

\vspace{7pt}

\emph{Step 1.}\quad Let $\eps,\delta > 0$, and let $(\s_{n_i})_{i\geq 1}$ and $(\mu_i)_{i\geq 1}$ be two sequences as in the definition of $\rmh^\m_\S(\mu)$.  By property (Pi), there exists $\P \in \frak{P}$ all of whose cells have diameter less than $\delta$.  Now any cell of $\P^{V_{n_i}}$ also has diameter less than $\delta$ according to $d^{(V_{n_i})}$, and therefore
\[\cov_{\eps,\delta}(\mu_i,d^{(V_{n_i})}) \leq \cov_\eps(\mu_i;\P^{V_{n_i}}).\]
This shows that $\rmh_\S^\m(\mu) \leq h$.

\vspace{7pt}

\emph{Step 2.}\quad Fix the sequences $(\s_{n_i})_{i\geq 1}$ and $(\mu_i)_{i\geq 1}$ and also $\eps > 0$ and $\P \in \frak{P}$. By property (Pii), we have
\[\mu_e\Big(\bigcup_{C \in \P}\rm{int}(C)\Big) = 1,\]
where $\rm{int}(C)$ denotes the interior of $C$.  Therefore there exist  $r > 0$ and further open subsets $G_C \subseteq C$ for each $C \in \P$ such that
\[d(x,x') > r \quad \hbox{whenever}\quad x \in G_C,\ x' \in G_{C'}\ \hbox{and}\ C\neq C',\]
and such that the open set
\[G := \bigcup_{C \in \P}G_C\]
has $\mu_e (G) > 1 - \eps^2/2$.  By the portmanteau theorem, it follows that the set
\[U := \big\{\nu \in \Pr(\X^G):\ \nu\{x:\ x_e \in G\} > 1 - \eps^2/2\big\}\]
is a weak$^\ast$ neighbourhood of $\mu$.

Next, observe that
\begin{align*}
\int |\{v\in V_{n_i}:\ x_v \in G\}|\,\mu_i(d\bf{x}) &= \sum_{v \in V_{n_i}}\mu_i\{\bf{x} \in \X^{V_{n_i}}:\ x_v \in G\}\\
&= \sum_{v \in V_{n_i}}((\Pi^{\s_{n_i}}_v)_\ast\mu_i)(\{x \in \X^G:\ x_e\in G\})\\
&> (1-\eps^2/2)|\{v\in V_{n_i}:\ (\Pi^{\s_{n_i}}_v)_\ast\mu_i \in U\}|.
\end{align*}
Since $\mu_i \lde \mu$, we have in particular that $\mu_i \lw \mu$.  Therefore the above integral is greater than $(1-\eps^2)|V_{n_i}|$ for all sufficiently large $i$.  By Markov's inequality, it follows that the sets
\[Z_i := \big\{\bf{x} \in\X^{V_{n_i}}:\ |\{v \in V_{n_i}:\ x_v \in G\}| > (1-\eps)|V_{n_i}|\big\}\]
satisfy $\mu_i(Z_i) > 1 - \eps$ for all sufficiently large $i$. Let
\[\Q_i := \{D\cap Z_i:\ D \in \P^{V_{n_i}}\}.\]

Now suppose that $\delta < \eps r$ and that $\bf{x},\bf{y} \in Z_i$ satisfy $d^{(V_{n_i})}(\bf{x},\bf{y}) < \delta$.  The fact that $\bf{x},\bf{y} \in Z_i$ gives that
\[|\{v\in V_{n_i}:\ x_v,y_v \in G\}| > (1-2\eps)|V_{n_i}|.\]
Combining this with the fact that $d^{(V_{n_i})}(\bf{x},\bf{y}) < \delta$ and Markov's inequality, we obtain
\begin{multline*}
\big|\big\{v\in V_{n_i}:\ x_v,y_v \ \hbox{lie in}\ G_C\ \hbox{for the same}\ C\in \P\big\}\big| \\ > (1-2\eps - \delta/r)|V_{n_i}| > (1-3\eps)|V_{n_i}|.
\end{multline*}
It follows that if $\bf{x} \in F \subseteq Z_i$, and if $F$ has diameter at most $\delta$ according to $d^{(V_{n_i})}$, then $F$ can have nonempty intersection with at most
\[2^{\rmH(3\eps,1-3\eps)|V_{n_i}|}\cdot |\P|^{3\eps|V_{n_i}|}\]
cells of $\Q_i$, since this is an upper bound on the number of elements of $\Q_i$ that agree with the cells containing $x_v$ for at least $(1-3\eps)|V_{n_i}|$ vertices $v$. (This is a standard bound on the cardinality of Hamming balls: see, for instance,~\cite[Section 5.4]{GolPin--book}.)

Finally, once $i$ is large enough that we have $\mu_i(Z_i) > 1 - \eps$, the last estimate implies that
\begin{align*}
&\log \cov_{\eps,\delta/2}(\mu_i,d^{(V_{n_i})}) \\
&\geq \log\min\big\{k:\ \exists F_1,\dots,F_k \subseteq Z_i\ \hbox{all having diameter $< \delta$}\\
& \qquad \qquad \qquad \qquad \qquad \qquad \hbox{and such that}\ \mu(F_1\cup\cdots\cup F_k) > 1-2\eps \big\}\\
&\geq \log \cov_{2\eps}(\mu_i;\P^{V_{n_i}}) - \rmH(3\eps,1-3\eps) |V_{n_i}| - 3\eps|V_{n_i}|\log|\P|.
\end{align*}
Since we may choose $\eps$ arbitrarily small, this implies that $\rmh_\S^\m(\mu) \geq h$.
\end{proof}

Next let us see that Theorem B can fail if one does not assume both properties (Pi) and (Pii).  Without property (Pi), it is easy to see from the proof that one could have $h < \rmh_\S^\m(\mu)$, where again $h$ denotes the right-hand side in the theorem.  More interestingly, without property (Pii), the reverse failure can occur.

\begin{ex}
Suppose that is $\X$ is an uncountable compact metric space which has no isolated points, and let $\frak{P}$ consist of \emph{all} finite Borel partitions of $\X$.  Suppose also that $\rmh_\S^\m(\mu) \geq 0$: this simply asserts that there are some subsequence $n_i \uparrow \infty$ and measures $\mu_i \in \Pr(\X^{V_{n_i}})$ which satisfy $\mu_i \lde \mu$ over $(\s_{n_i})_{i\geq 1}$.  Let us show that in this case the right-hand formulae of Theorem B are both equal to $+\infty$, irrespective of the exact value of $\rmh_\S^\m(\mu)$.

By the Baire category theorem, any co-countable subset of $\X$ is dense.  Using this fact, for any $k\in\bbN$, a simple recursion produces a Borel partition
\[\P= (C_1,\dots,C_{k-1},R)\]
in which $C_1$, \dots, $C_{k-1}$ are all countable dense subsets of $\X$, and the remainder set $R$ is also dense.

It follows that every cell of $\P^{V_{n_i}}$ is dense in $\X^{V_{n_i}}$ for every $i$. Therefore we can make a sequence of smaller and smaller perturbations $\nu_i$ of $\mu_i$ in the weak$^\ast$ topologies so that, on the one hand, we still have $\nu_i \lde \mu$; but, on the other, each $\nu_i$ gives equal weight to every cell of $\P^{V_{n_i}}$, and hence $\rmH_{\nu_i}(\P^{V_{n_i}}) = \log k\cdot|V_{n_i}|$.  It follows that the second right-hand quantity in Theorem B is greater than $\log k$ for every $k \in \bbN$, and so must be infinite. Since Proposition~\ref{prop:mment} still applies, this is also true of the first right-hand quantity. \fin
\end{ex}

\subsection{Related results for local weak$^\ast$ convergence}

Natural variants of model-measure sofic entropy can be obtained by replacing LDE convergence in the definition with other modes of convergence.  The paper~\cite{Aus--soficentadd} already includes a notion which uses only LE convergence; there it was called `quenched model-measure sofic entropy', and here we denote it by $\rmh_\S^{\rm{le}}$.  Most of these variants probably have little value for ergodic theory, but some may have an important auxiliary role in computing values of other, existing entropy notions.  For instance, the quantity $\t{\rmh}_\S(\mu)$ of Subsection~\ref{subs:coind} plays such a role through Theorem C.

Here we consider replacing LDE convergence with local weak$^\ast$ convergence.  We have already seen that Theorem A has no analog in that setting. Some parts of the previous subsection do still work, but others fail. Of course, problems appear only in case $\mu$ is not weakly mixing, since all these modes of convergence coincide for weakly mixing processes.

In the proof of Proposition~\ref{prop:mment}, Step 1 makes no reference to the specific mode of convergence, so it still applies without change.  On the other hand, Step 2 relies on Theorem A, which is not available for local weak$^\ast$ convergence in case $\mu$ is not weakly mixing, as seen in Example~\ref{ex:lw-bad}.  So in the case of local weak$^\ast$ convergence we obtain only an inequality.

\begin{prop}\label{prop:lwent}
	Let $(\X^G,\mu,S,d)$ be a metric $G$-process and let $\P$ be a finite measurable partition of $\X$.  Then
	\begin{align*}
		&\sup\Big\{\limsup_{i\to\infty} \frac{1}{|V_{n_i}|}\log \cov_\eps(\mu_i;\P^{V_{n_i}}): \eps > 0,\\ &
		\qquad \qquad \qquad \qquad \qquad \qquad   n_i\uparrow \infty \ \hbox{and}\ \mu_i \lw \mu\ \hbox{over}\ (\s_{n_i})_{i\geq 1}\Big\}\\
		&\geq \sup\Big\{\limsup_{i\to\infty} \frac{1}{|V_{n_i}|}\rmH_{\mu_i}(\P^{V_{n_i}}):\ n_i\uparrow \infty \ \hbox{and}\ \mu_i \lw \mu\ \hbox{over}\ (\s_{n_i})_{i\geq 1}\Big\}.
	\end{align*}
	\qed
\end{prop}

Similarly, the argument above for Theorem B uses only the convergence $\mu_i \lw \mu$, not the full strength of the assumption that $\mu_i \lde \mu$. So, together with the above proposition, that proof can be re-applied for local weak$^\ast$ convergence. This gives a counterpart of Theorem B which again provides only an inequality.

\begin{thm}\label{thm:lwent}
	Let $(\X^G,\mu,S,d)$ be a metric $G$-process and let $\frak{P}$ be any family of finite Borel partitions of $\X$ having properties (Pi) and (Pii), as in Theorem B.  Then
	\begin{align*}
		&\sup\Big\{\sup_{\eps,\delta > 0}\limsup_{i\to\infty} \frac{1}{|V_{n_i}|}\log \cov_{\eps,\delta}(\mu_i,d^{(V_{n_i})}):\\ & \qquad \qquad \qquad \qquad \qquad \qquad n_i\uparrow \infty \ \hbox{and}\ \mu_i \lw \mu\ \hbox{over}\ (\s_{n_i})_{i\geq 1}\Big\}\\
		&= \sup\Big\{\limsup_{i\to\infty} \frac{1}{|V_{n_i}|}\log \cov_\eps(\mu_i;\P^{V_{n_i}}): \eps > 0,\ \P \in \frak{P},\\ & \qquad \qquad \qquad \qquad \qquad \qquad n_i\uparrow \infty \ \hbox{and}\ \mu_i \lw \mu\ \hbox{over}\ (\s_{n_i})_{i\geq 1}\Big\}\\
		&\geq \sup\Big\{\limsup_{i\to\infty} \frac{1}{|V_{n_i}|}\rmH_{\mu_i}(\P^{V_{n_i}}): \P \in \frak{P},\\ & \qquad \qquad \qquad \qquad \qquad \qquad  n_i\uparrow \infty \ \hbox{and}\ \mu_i \lw \mu\ \hbox{over}\ (\s_{n_i})_{i\geq 1}\Big\}.
	\end{align*}
	\qed
\end{thm}

The last quantity in Theorem~\ref{thm:lwent} is $\t{\rmh}_\S(\mu)$ again.  The inequality can be strict: Example~\ref{ex:lw-bad} demonstrates this, too.  By slightly adapting the proof of~\cite[Theorem 6.4]{Aus--soficentadd}, it is not hard to show that the first two formulae above give another quantity which is invariant under measure-theoretic isomorphism.  We omit the details since I do not currently know of any applications for this quantity.  Moreover, I believe it can be related to existing invariants as follows.  For $\mu_i \in \Pr(\X^{V_{n_i}})$, we write
\[P^{\s_{n_i}}_\ast\mu_i \in \Pr(\Pr(\X^G))\]
for the pushforward of $\mu_i$ under the empirical-distribution map $\bf{x}\mapsto P^{\s_{n_i}}_{\bf{x}}$.  According to~\cite[Lemma 5.6]{Aus--soficentadd}, any weak$^\ast$ limit $\theta$ of the sequence $(P^{\s_{n_i}}_\ast\mu_i)_{i\geq 1}$ must be supported on $\Pr^G(\X^G)$ and have barycentre equal to $\mu$.  Let us pass to a subsequence so that in fact $P^{\s_{n_i}}_\ast\mu_i \stackrel{\rm{weak}^\ast}{\to} \theta$. Then it should follow that the first two formulae of Theorem~\ref{thm:lwent} are equal to the essential supremum of $\rmh^{\rm{le}}_\S(\nu)$ when $\nu$ is drawn from the distribution $\theta$. In particular, if $\mu$ is ergodic, then these two formulae simply compute $\rmh_\S^{\rm{le}}(\mu)$, as can be deduced directly from~\cite[Corollary 5.7]{Aus--soficentadd}.

These facts would give a relative of some older work~\cite{Wink64,Wink70} in information theory which addresses the following situation. Let $(\X^\bbZ,\mu,S)$ be a stationary source with a finite alphabet, and let $\mu_n$ be the marginal of $\mu$ on $\X^n$ for each $n$.  If $\mu$ is not ergodic, then the AEP can fail (put another way: the Shannon-McMillan theorem does not hold in its usual form), and the KS-entropy need not be equal to the exponential growth-rate of the covering numbers of the measures $\mu_n$.  In this case the KS-entropy is the average of the KS-entropies of the ergodic components, whereas the growth of the covering numbers is governed by the full distribution of those entropy values over the ergodic decomposition of $\mu$. In particular, the exponential growth rate of $\rm{cov}_\eps(\mu_n)$ approaches the essential supremum of the KS-entropies of the ergodic components as $\eps \downarrow 0$.  See also~\cite{Wink77} for a generalization to processes with arbitrary state spaces and a resulting non-ergodic generalization of Krieger's generator theorem; and~\cite{Suj80} for some more analysis of the distribution of KS-entropies over ergodic components, with an application to a source coding theorem.

\subsection{Alternative approaches for general metric processes}

In this paper, families of partitions which satisfy properties (Pi) and (Pii) are simply a convenient way to bring Shannon entropy and the AEP to bear when studying processes with non-finite state spaces.  Questions about the appropriate `entropy notion' when faced with such general state spaces have a long history in information theory.  The traditional setting is shift-invariant measures on $\X^\bbZ$ for some standard measurable space $\X$.  In that case the classical Shannon--McMillan theorem becomes available upon suitably quantizing $\X$, although one must be careful about taking limits in the right order. See~\cite{GraKie80}, for example, which also approaches quantization by starting with a choice of suitable sequences of partitions.

I expect relatives of Theorem B can be found which do not start with families satisfying (Pi) and (Pii).  But two issues arise immediately: choosing a suitable `metric' notion of Shannon entropy; and making sense of the AEP without reference to a particular partition.

To address the first issue, one could seek a reformulation of Theorem B using an appropriate version of `epsilon Shannon entropy' for a metric probability space $(X,d,\mu)$, such as the quantity
\begin{align*}
\rmH_{\eps,\delta}(d,\mu) &:= \inf\Big\{\rmH_\mu(\{U_0,U_1,\dots,U_k\}):\\ & \qquad \qquad U_0,\dots,U_k\ \hbox{are a Borel partition of $X$,}\\
& \qquad \qquad \mu(U_0) < \eps\ \hbox{and}\ \rm{diam}(U_i,d) < \delta\ \forall i = 1,2,\dots,k\Big\}
\end{align*}
for $\eps,\delta > 0$. This still involves partitions of the metric space, but they are hidden inside a natural `Shannon-like' quantity which depends only on the metric space itself.

Such quantities have their own place in information theory. In particular,~\cite{PosRodRum67,PosRod72,Katet86} develop the basic properties of a quantity very similar to $\rmH_{\eps,\delta}(d,\mu)$, and involve some similar estimates to those in the present paper.  Vershik has also explored notions with this flavour in the setting of ergodic theory: see, for instance, the end of Section 4 in~\cite{Ver04} or~\cite[Section 6]{Ver00}.

Some of those works do study behaviour under Cartesian products, but I do not know of an established metric-space version of the AEP.  In our setting, one candidate is as follows.  Given a compact metric space $(\X,d)$, a sequence $\mu_n \in \Pr(\X^{V_n})$ could satisfy the `metric AEP with rate $h$' if for every $\eta > 0$ and every $\eps \in (0,1/2)$ there exists $\delta_0 > 0$ such that
\[\limsup_{n\to\infty}\Big|\frac{1}{|V_n|}\log \min \Big\{|F|:\ F\subseteq \X^{V_n},\ \mu_n(B_\delta(F)) > a\Big\} - h\Big| < \eta\]
for all $\delta \in (0,\delta_0)$ and all $a \in (\eps,1-\eps)$, where $B_\delta(F)$ is the $\delta$-neighbourhood of $F$ according to the metric $d^{(V_n)}$.  Informally, for any radius parameter $\delta$ less than $\delta_0$, this asserts that it takes roughly $\exp(h|V_n|)$ balls of radius $\delta$ to cover $\mu_n$, whether one wishes to cover 5\% or 95\% of this measure.

In case $\X$ is finite, it is not hard to deduce the above property from the usual AEP.  However, I have not explored further how this metric definition is related to the other problems of the present paper in the case of general $(\X,d)$.  It is not clear to me that this approach gives any real advantage over the more explicit use of partition-families that we have adopted above.

\section{Application to co-induced systems}\label{sec:coind}

This section proves Theorem C. Fix a metric $G$-process $(\X^G,\mu,S,d)$, so that the co-induced system is the metric ${(G\times H)}$-process $(\X^{G\times H},\mu^{\times H},S,d)$.  Let $\frak{P}$ be a collection of Borel partitions of $\X$ having the two properties (Pi) and (Pii) required by Theorem B.

In this section we use the following notation.  If $\mu \in \Pr(\X^G)$ and $E \subseteq G$, then $\mu_E$ denotes the marginal of $\mu$ on $\X^E$: that is, the image measure of $\mu$ under the coordinate projection map $\X^G\to \X^E$.  Similarly, if $\mu \in \Pr(\X^V)$ for some finite set $V$ and $U \subseteq V$, then $\mu_U$ denotes the marginal of $\mu$ on $\X^U$.

As in Subsection~\ref{subs:coind}, let us write $\t{\rmh}_\S(\mu)$ for the right-hand formula in Theorem C. We must show that it equals $\rmh^\m_{\S\times \rm{T}}(\mu^{\times H})$. The proof is based on Theorem B, together with two simple lemmas.

\begin{lem}\label{lem:conv-and-conv1}
	If $\mu_n\lw \mu$ over $\S$, then $\mu_n^{\times W_n}\lde \mu^{\times H}$ over $\S\times \rm{T}$.
\end{lem}

\begin{proof}\emph{Part 1.}\quad We first show that $\mu_n^{\times W_n} \stackrel{\rm{lw}^\ast}{\to} \mu^{\times H}$. Let $E\subseteq G$ and $F\subseteq H$ both be finite.  Once $n$ is sufficiently large, most vertices $w \in W_n$ have the property that the map $h\mapsto \tau_n^h(w)$ is injective on $F$.  For such $w$, and any $v \in V_n$, a simple calculation gives
	\[\big((\Pi^{\s_n\times \tau_n}_{(v,w)})_\ast(\mu_n^{\times W_n})\big)_{E\times F} = \big(((\Pi^{\s_n}_v)_\ast \mu_n)_E\big)^{\times F}.\]
Once $n$ is sufficiently large, most $v \in V_n$ have the property that $((\Pi^{\s_n}_v)_\ast\mu_n)_E$ is close to $\mu_E$ in the weak$^\ast$ topology, and hence the above measure is close to $\mu_E^{\times F}$.

\vspace{7pt}

\emph{Part 2.}\quad If $H$ is infinite, then the shift action of $G\times H$ on $(\X^{G\times H},\mu^{\times H})$ is weakly mixing.  In this case local weak$^\ast$ convergence implies LDE convergence~\cite[Corollary 5.7 and Lemma 5.15]{Aus--soficentadd}, so Part 1 completes the proof.

So now suppose that $H$ is finite.  In this case we must use the weaker assumption that $|W_n| \to \infty$ to complete the proof.  There are many ways to do this.  Here we use a simple trick of enlarging $H$ to the infinite group $\t{H} := H\times \bbZ$ without effectively changing the sofic approximation to $H$.

Indeed, since $H$ is finite and $\rm{T}$ is a sofic approximation to it, the sets 
\begin{multline*}
Y_n := \big\{w \in W_n:\ \tau_n^g(\tau_n^h(w)) = \tau_n^{gh}(w)\ \forall g,h \in H\\ \hbox{and}\ \tau_n^h(w) \neq w\ \forall h \in H\setminus \{e_H\}\big\}
\end{multline*}
must satisfy $|Y_n| = (1-o(1))|W_n|$.  By the definition of these sets, the restrictions $h\mapsto \tau_n^h|Y_n$ define a genuine action of $H$ on $Y_n$ for each $n$, and these actions are all free.

Since $H$ is finite but $|Y_n| \to \infty$, it follows that $Y_n$ must consist of a disjoint union of free $H$-orbits, their number tending to $\infty$ with $n$.  We may therefore choose a sequence of permutations $\theta_n:Y_n\to Y_n$ such that each $\theta_n$ commutes with the $H$-action on $Y_n$ and permutes the $H$-orbits cyclically.  Extending $\theta_n$ arbitrarily to a permutation of $W_n$, it follows that the maps
\[\t{\tau}_n:\t{H}\to \rm{Sym}(W_n):(h,k)\mapsto \tau_n^h\circ \theta_n^k\]
constitute a sofic approximation to $\t{H}$.  Let $\t{\rm{T}} := (\t{\tau}_n)_{n\geq 1}$.

Now Part 1 still applies with the new group $\t{H}$ and sofic approximation $\t{\rm{T}}$, giving that $\mu_n^{\times W_n} \lw \mu^{\times\t{H}}$ over $\S\times \t{\rm{T}}$.  Since $\t{H}$ is infinite, this implies that also $\mu_n^{\times W_n}\lde \mu^{\times \t{H}}$ over $\S\times \t{\rm{T}}$, as explained above.  Finally, applying the coordinate projection map $\X^{G\times \t{H}} \to \X^{G\times H}$ to $\mu^{\times \t{H}}$ and to the empirical measures of models chosen randomly from $\mu^{\times W_n}$, this implies that $\mu^{\times W_n}\lde \mu^{\times H}$ over $\S\times \rm{T}$.
\end{proof}

\begin{rmk}
Recalling the definitions of LDE convergence and of the weak$^\ast$ topology on measures, Part 2 of the previous proof is equivalent to proving that
\[(\mu_n^{\times W_n}\times \mu_n^{\times W_n})\Big\{(\bf{x},\bf{y}):\ \Big|\int f\,dP^{\s_n\times \tau_n}_{(\bf{x},\bf{y})} - \int f\,d(\mu^{\times H}\times \mu^{\times H})\Big| < \eps\Big\} \to 1\]
as $n\to\infty$ for any $f \in C(\X^{G\times H}\times \X^{G\times H})$ and any $\eps > 0$.  Rather than using the trick above to reduce to the case of infinite $H$, it is also easy to prove this directly, just a little messier.  First, by the Stone--Weierstrass theorem, it suffices to assume that $f$ depends on only finitely many coordinates in $\X^{G\times H}\times \X^{G\times H}$.  Having done so, the integral $\int f\,dP^{\s_n\times \tau_n}_{(\bf{x},\bf{y})}$ may be written as an average over $V_n\times W_n$ of quantities that depend only on local `patches' of the models $\bf{x}$ and $\bf{y}$.  Using the product structure of $\mu_n^{\times W_n}$, one deduces that most pairs of terms in this average are independent once $n$ is large. Then a simple appeal to Chebychev's inequality gives the convergence in probability asserted above, much as in standard proofs of the weak law of large numbers. \fin
\end{rmk}

\begin{lem}\label{lem:conv-and-conv2}
	If $\nu_n\lde \mu^{\times H}$ over $\S\times \rm{T}$, then there are subsets $Z_n \subseteq W_n$ such that
\[|Z_n| = (1-o(1))|W_n|\]
and such that, for any selection of a sequence of vertices $w_n \in Z_n$, we have
\[(\nu_n)_{V_n\times \{w_n\}} \lw \mu \quad \hbox{over}\ \S.\]
\end{lem}

\begin{proof}
Let $U_1 \supseteq U_2 \supseteq \dots$ be a neighbourhood base for the weak$^\ast$ topology around $\mu \in \Pr(\X^G)$.  Let $e_H$ be the identity of $H$, and let us commit the slight abuse of identifying measures on $\X^{G\times \{e_H\}}$ and on $\X^G$. Since $\nu_n \lde \mu^{\times H}$, we have
\begin{align*}
&\big|\big\{(v,w) \in V_n\times W_n:\ ((\Pi^{\s_n\times \tau_n}_{(v,w)})_\ast \nu_n)_{G\times \{e_H\}} \in U_k\big\}\big|\\
&= \big|\big\{(v,w) \in V_n\times W_n:\ (\Pi^{\s_n}_v)_\ast((\nu_n)_{V_n\times \{w\}}) \in U_k\big\}\big|\\
&=\sum_{w \in W_n}\big|\big\{v \in V_n:\ (\Pi^{\s_n}_v)_\ast((\nu_n)_{V_n\times \{w\}}) \in U_k\big\}\big|\}\\
&= (1-o(1))|V_n\times W_n|
\end{align*}
for every $k \geq 1$. Therefore, by Markov's inequality, if we choose $k_1 \leq k_2 \leq \dots$ growing to $\infty$ sufficiently slowly and also $\eps_n \downarrow 0$ sufficiently slowly, then the sets
\[Z_n := \Big\{w \in W_n:\ \big|\big\{v \in V_n:\ (\Pi^{\s_n}_v)_\ast((\nu_n)_{V_n\times \{w\}}) \in U_{k_n}\big\}\big| > (1-\eps_n)|V_n|\Big\}\]
satisfy $|Z_n| = (1-o(1))|W_n|$.
\end{proof}

\begin{proof}[Proof of Theorem C]
	\emph{Part 1.}\quad Let $\P \in \frak{P}$, let $n_i \uparrow \infty$, and let $\mu_i \in \Pr(\X^{V_{n_i}})$ be a sequence which locally weak$^\ast$ converges to $\mu$ over $(\s_{n_i})_{i\geq 1}$.
	
	By Lemma~\ref{lem:conv-and-conv1}, it also holds that $\mu_i^{\times W_n}\lde \mu^{\times H}$ over ${(\s_{n_i}\times \tau_{n_i})_{i \geq 1}}$.  On the other hand, the additivity of Shannon entropy gives
	\begin{equation}\label{eq:Sh-add}
\frac{1}{|V_{n_i}|}\rmH_{\mu_i}(\P^{V_{n_i}}) = \frac{1}{|V_{n_i}\times W_{n_i}|}\rmH_{\mu_i^{\times W_{n_i}}}(\P^{V_{n_i}\times W_{n_i}}).
\end{equation}
	Taking the limit supremum as $i\to\infty$, the second formula from Theorem B turns this equality into
\[\limsup_{i\to\infty} \frac{1}{|V_{n_i}|}\rmH_{\mu_i}(\P^{V_{n_i}}) \leq \rmh^\m_{\S\times \rm{T}}(\mu^{\times H}).\]
Now taking the supremum over all such choices of $\P \in \frak{P}$, $(n_i)_{i\geq 1}$, and $(\mu_i)_{i\geq 1}$, this becomes $\t{\rmh}_\S(\mu) \leq \rmh_{\S\times \rm{T}}^\m(\mu^{\times H})$.

\vspace{7pt}

\emph{Part 2.}\quad Let $\P \in \frak{P}$, let $n_i\uparrow \infty$, and let $\nu_i \in \Pr(\X^{V_{n_i}\times W_{n_i}})$ be a sequence which LDE converges to $\mu^{\times H}$ over $(\s_{n_i}\times \tau_{n_i})_{i\geq 1}$.

Let $Z_i\subseteq W_{n_i}$ be a sequence of subsets as provided by Lemma~\ref{lem:conv-and-conv2}.  For any selection of a sequence of vertices $w_i \in Z_i$, we have
\[(\nu_i)_{V_{n_i}\times \{w_i\}} \lw \mu \quad \hbox{over}\ (\s_{n_i})_{i\geq 1},\]
and therefore
\[\limsup_{i\to\infty} \frac{1}{|V_{n_i}|}\rmH_{(\nu_i)_{V_{n_i}\times\{w_i\}}}(\P^{V_{n_i}}) \leq \t{\rmh}_\S(\mu),\]
by the definition of $\t{\rmh}_\S(\mu)$. Since we can choose the elements $w_i \in Z_i$ to maximize the Shannon entropies appearing on the left-hand side here, we must actually have
\begin{equation}\label{eq:max-ent-bd}
\limsup_{i\to\infty} \frac{1}{|V_{n_i}|}\max_{w \in Z_i}\rmH_{(\nu_i)_{V_{n_i}\times\{w\}}}(\P^{V_{n_i}}) \leq \t{\rmh}_\S(\mu).
\end{equation}

Now subadditivity and standard bounds for Shannon entropy give
\begin{align*}
&\rmH_{\nu_i}(\P^{V_{n_i}\times W_{n_i}}) \\ &\leq \sum_{w \in W_{n_i}}\rmH_{(\nu_i)_{V_{n_i}\times \{w\}}}(\P^{V_{n_i}})\\
&\leq \sum_{w \in Z_i}\rmH_{(\nu_i)_{V_{n_i}\times \{w\}}}(\P^{V_{n_i}}) + |W_{n_i}\setminus Z_i|\cdot |V_{n_i}|\cdot \log|\P|\\
&\leq |W_{n_i}|\cdot \max_{w \in Z_{n_i}}\rmH_{(\nu_i)_{V_{n_i}\times \{w\}}}(\P^{V_{n_i}}) + o(|V_{n_i}\times W_{n_i}|)\cdot \log|\P|,
\end{align*}
and so~\eqref{eq:max-ent-bd} gives
\[\limsup_{i\to\infty}\frac{1}{|V_{n_i}\times W_{n_i}|}\rmH_{\nu_i}(\P^{V_{n_i}\times W_{n_i}}) \leq \t{\rmh}_\S(\mu).\]
Taking the supremum over all choices of $\P \in \frak{P}$, $(n_i)_{i\geq 1}$, and $(\nu_i)_{i\geq 1}$, one final appeal to the second formula in Theorem B gives $\rmh_{\S\times \rm{T}}^\m(\mu^{\times H}) \leq \t{\rmh}_\S(\mu)$.
\end{proof}

Through its appeals to the second formula in Theorem B, the proof above rests crucially on the use of measures $\mu_i$ or $\nu_i$ that have the AEP.  This is what allows us to compute the sofic entropies using $\rmH_{\mu_i}(\P^{V_{n_i}})$ in place of $\cov_\eps(\mu_i;\P^{V_{n_i}})$, and so gives us access to the additivity in~(\ref{eq:Sh-add}).  Covering numbers of more general measures need not behave so well under high-dimensional Cartesian products.




\subsubsection*{Acknowledgements}

This research was supported partly by the Simons Collaboration on Algorithms and Geometry. \fin

\bibliographystyle{abbrv}
\bibliography{bibfile}

\parskip 0pt
\parindent 0pt

\vspace{7pt}

%
%
%




\end{document}